\documentclass[12pt]{amsart}
\usepackage{amsmath}
\usepackage{amsfonts}
\usepackage[mathscr]{eucal}
\usepackage{amssymb}
\usepackage{mathtools, xypic, color}
\usepackage{enumitem}
\usepackage{hyperref}
\usepackage{xcolor}

\newtheorem{thm}{Theorem}[section]

\newtheorem{cor}[thm]{Corollary}

\newtheorem{lemma}[thm]{Lemma}

\theoremstyle{definition}

\newtheorem{remark}[thm]{Remark}

\mathtoolsset{centercolon}

\newcommand{\bb}[1]{\mathbb{#1}}
\newcommand{\cl}[1]{\mathcal{#1}}

\newcommand{\rC}{\mathrm{C}}
\newcommand{\ep}{\varepsilon}
\renewcommand{\phi}{\varphi}

\newcommand{\Ad}{\operatorname{Ad}}

\newcommand{\ol}{\overline}

\newcommand{\ip}[1]{\langle #1 \rangle}
\newenvironment{sbmatrix}{\left[\begin{smallmatrix}}{\end{smallmatrix}\right]}

\newcommand{\qand}{\quad\text{and}\quad}

\newcommand{\qfor}{\quad\text{for}\quad}
\newcommand{\qforal}{\quad\text{for all}\quad}

\begin{document}

\title[Complete spectral sets and numerical range]{Complete spectral sets\\ and numerical range}

\author[K.R. Davidson]{Kenneth R. Davidson}
\address{Pure Mathematics Dept., University of Waterloo,
Waterloo, ON, N2L 3G1, Canada}
\email{krdavids@uwaterloo.ca \vspace{-2ex}}
\thanks{The first author is partially supported by an NSERC grant.}

\author[V.I.~Paulsen]{Vern~I.~Paulsen}
\email{vpaulsen@uwaterloo.ca}
\thanks{The second author is partially supported by an NSERC grant.}

\author[H.J. Woerdeman]{Hugo J.~Woerdeman}
\address{Department of Mathematics \\
Drexel University\\
Philadelphia, PA, 19104\\
U.S.A.}
\email{hugo@math.drexel.edu}
\thanks{The third author is partially supported by a Simons Foundation grant and by the Institute for Quantum Computing at the University of Waterloo.}


\begin{abstract}
We define the complete numerical radius norm for homomorphisms from any operator algebra into $\cl B(\cl H)$, and show that this norm can be computed explicitly in terms of the completely bounded norm. This is used to show that if $K$ is a complete $C$-spectral set for an operator $T$, then it is a complete $M$-numerical radius set, where $M=\frac12(C+C^{-1})$. In particular, in view of Crouzeix's theorem, there is a universal constant $M$ (less than 5.6) so that if $P$ is a matrix polynomial and $T \in \cl B(\cl H)$, then $w(P(T)) \le M \|P\|_{W(T)}$. When $W(T) = \ol{\bb D}$, we have $M = \frac54$.
\end{abstract}

\maketitle

\stepcounter{section}

In 2007, Michel Crouzeix \cite{Crouzeix} proved the remarkable fact that for any operator $T$ on a Hilbert space $\cl{H}$, 
the numerical range is a complete $C$-spectral set for some constant with a universal bound of 11.08.
Moreover in \cite{Crouzeix2}, he conjectures that the optimal constant is $2$, which is the case for a disc. 
This inspired a result of Drury \cite{Drury} who proved that if the numerical range of $T$ is contained in the disc, then the numerical radius of any polynomial in $T$ is bounded by $\frac54$ times the supremum norm of the polynomial over the disc. 
Generally all that one can say about the relationship between the norm and numerical radius is that $w(X) \le \|X\|$, with equality for many operators,  so the improvement from $2$ to Drury's $\frac54$ was unexpected.

In this note, we establish a precise relationship between the completely bounded norm of a homomorphism of an arbitrary operator algebra and what we call the {\it complete numerical radius norm} of the homomorphism.  When applied to the case of the disc, our relationship yields Drury's result in the matrix polynomial case, and our result also applies to the general study of $C$-spectral sets.

We now introduce the definitions and prior results that we shall need.


A set $K$ is a {\it $C$-spectral set} for an operator $T$ if for every rational function $f$ with poles off of $K$, one has
\[ \| f(T)\| \le C \|f\|_K  \]
where $\|f\|_K = \sup\{ |f(z)| : z \in K\}$. 
It is a {\it complete $C$-spectral set} if this inequality holds for all matrices with rational coefficients.
This inequality clearly extends to $R(K)$, the uniform closure of these rational functions in $\rC(K)$.
It is well-known that when the complement of $K$ is connected (such as when $K$ is convex), it suffices to verify this inequality for polynomials.
When the constant $C=1$, we call $K$ a {\it (complete) spectral set}.
The second author \cite{Paulsen84} showed in 1984 that $K$ is a complete $C$-spectral set for $T$ if and only if there is an invertible operator $S$ so that $K$ is a complete spectral set for $STS^{-1}$ and $\|S\|\,\|S^{-1}\| \le C$.

The numerical range of an operator $T$ is the set
\[ W(T) = \{ \ip{Tx,x} : \|x\|=1 \} .\]
This set is always convex and contains the spectrum of $T$. The numerical radius of $T$ is
\[ w(T) = \sup \{ |\lambda| : \lambda\in W(T) \} .\]
It is well known that $w(T) \le \|T\| \le 2 w(T)$, and that this inequality is sharp.

Early estimates for $w(p(T))$ for a polynomial $p$ in the case of the disc were due to Berger and Stampfli \cite{BS}:
if $w(T) \le 1$ and $p(0)=0$, then $w(p(T)) \le \|p\|_{\bb{D}}$. A recent result of Drury \cite{Drury} deals with the case of $p(0) \ne 0$:
\[ w(p(T)) \le \frac54 \|p\|_{\bb{D}} .\]
This inequality can be seen to be sharp for $T = \begin{sbmatrix}0&2\\0&0\end{sbmatrix}$ and the M\"obius map that takes $0$ to $1/2$.
Drury gives more precise information about the shape of $W(p(T))$ as a ``teardrop". For another proof of this fact, see \cite{KMR}.
Since we are primarily interested in matrix polynomials, this geometric picture is no longer valid, but we prove the same $\frac54$ inequality.

We will say that a compact set $K$ is a {\it $C'$-numerical radius set} for $T$ if 
\[ w(f(T)) \le C' \|f\|_K \qforal f \in R(K) ,\]
and it is a {\it complete $C'$-numerical radius set} if the same inequality holds for matrices over $R(K)$. 
One of our key results is the following:
\\

{\bf Theorem 3.1.}
{\it Let $C\ge1$ and set $C'=\frac12(C+C^{-1})$.
A compact subset $K\subset \bb C$ is a complete $C$-spectral set for $T\in \cl B(\cl H)$ 
if and only if it is a complete $C'$-numerical radius set for $T$.}

In fact, we will prove a general result about unital operator algebras. 
Recall that every unital operator algebra $\cl A$ has a family of norms on $\cl M_n(\cl A)$, and that $\cl A$ may be represented completely isometrically as an algebra of operators on some Hilbert space. See \cite{Paulsen} for details.
If $\Phi:\cl A \to \cl B(\cl H)$ is a bounded linear map, it induces coordinatewise maps 
$\Phi^{(n)}:\cl M_n(\cl A) \to \cl M_n(\cl B(\cl H)) \simeq \cl B(\cl H^{(n)})$; and one defines the {\it completely bounded norm} by
\[ \|\Phi\|_{cb} = \sup_{n\ge1} \|\Phi^{(n)}\| .\]
We will also define a {\it complete numerical radius norm} on such maps
\[ \|\Phi\|_{wcb} := \sup_{n\ge1} \sup_{A \in \cl M_n(\cl A),\ \|A\|\le 1} w(\Phi^{(n)}(A)) .\]
Our main result is the following:
\\

{\bf Theorem 2.3.}
{\it Let $\cl A$ be an operator algebra, and let $\Phi$ be a completely bounded homomorphism. Then
\[ \|\Phi\|_{wcb} = \frac12 \big( \|\Phi\|_{cb} + \|\Phi\|_{cb}^{-1} \big) .\]}

\bigbreak
We complete this introduction with a bit more background on the numerical radius.
Ando \cite{Ando} provided some useful characterizations of $w(T)$. One we require is Ando's numerical radius formula:
\[ w(T) =  \min \Big\{ \tfrac12 \|A+B\| : \begin{bmatrix}A&T\\T^*&B\end{bmatrix} \ge 0 \Big\} . \]

Numerical range is intimately connected to dilation theory.
The first such result was due to Berger \cite{Berger}, who showed that if $w(T) \le 1$, then there is a unitary operator $U$ on a Hilbert space $\cl{K}$ containing $\cl{H}$ such that
\[ T^n = 2 P_\cl{H} U^n |_\cl{H} \qfor n\ge1 .\]
Easy examples show that the converse is not valid. 
Sz.Nagy and Foia\c{s} \cite{SNF} introduced the notion of a $\rho$-contraction (for $\rho \ge1$) as an operator $T$ for which there is a $\rho$-dilation, meaning a unitary operator $U$ on $\cl K \supset\cl H$ such that
\[ T^n = \rho P_\cl{H} U^n |_\cl{H} \qfor n\ge1 .\]
Okubo and Ando \cite{OA} show that if $T$ is a $\rho$-contraction, then  there is an invertible $S$ so that $\|STS^{-1}\| \le 1$ and $\|S\|\,\|S^{-1}\| \le \rho$.
By the remarks above, this shows that $\ol{\bb D}$ is a complete $\rho$-spectral set for $T$.
In particular, if $w(T)\le 1$, then $T$ is a 2-contraction by Berger's dilation, and thus $\ol{\bb D}$ is a complete 2-spectral set for $T$.

In view of the work of Crouzeix \cite{Crouzeix}, there has been a lot of renewed interest in numerical range.
See the monograph \cite{GR} and the recent survey  \cite{BBsurvey} for many relevant references.

\section{The main theorem}

We begin with a key observation which yields one direction of our theorem.

\begin{lemma} \label{L:num_range}
If $\|T\| \le 1$ and $\|S\| \, \|S^{-1}\| \le C$, then 
\[ w(S^{-1}TS) \le \frac12(C+C^{-1}) .\]
\end{lemma}

\begin{proof}
Using polar decomposition, $S=U|S|$, we may replace $T$ by the unitarily equivalent $U^*TU$ and suppose that $S>0$.
After scaling, we may suppose that $C^{-1/2}I \le S \le C^{1/2}I$.
Since $\|T\|\le1$ we have that
\[
 \begin{bmatrix}0 &0\\0&0 \end{bmatrix} \le
 \begin{bmatrix}S^{-1} &0\\0& S \end{bmatrix}
 \begin{bmatrix}I & T\\ T^* & I \end{bmatrix}
 \begin{bmatrix}S^{-1} &0\\0& S \end{bmatrix} =
 \begin{bmatrix}S^{-2} & S^{-1}TS\\ ST^*S^{-1}& S^2 \end{bmatrix}
\]
By  Ando's numerical radius formula, we obtain that
\begin{align*}
  w(S^{-1}TS) &\le \frac12 \| S^{-2} + S^2 \| \\
  &\le \sup \{ \frac12(t+t^{-1}) : C^{-1} \le t \le C\} \\
  &= \frac12(C+C^{-1}) . \qedhere
\end{align*}
\end{proof}

To establish the converse, we first need a simple computational lemma.

\begin{lemma} \label{L:2x2}
Let $B \in \cl B(\cl H)$ and let $T = \begin{bmatrix}\alpha I & B\\0& \alpha I \end{bmatrix} $. Then
\[ 
\|T\| = \frac{ \|B\| + \sqrt{\|B\|^2 + 4 |\alpha|^2}}2 \qand w(T) = |\alpha| +\tfrac12 \|B\|.
\] 
In particular, $\|T\|=1$ if and only if $|\alpha|^2+\|B\|=1$.
\end{lemma}

\begin{proof}
It is straightforward to show that $\|T\| = \left\| \begin{sbmatrix} \|B\|&|\alpha| \\ |\alpha| &0 \end{sbmatrix} \right\|$,
and computation of the eigenvalues of this self-adjoint matrix yields the desired formula. Routine manipulation now shows that
$\|T\|=1$ if and only if $|\alpha|^2+\|B\|=1$. 
It is also easy to see that $W\left(  \begin{sbmatrix} 0&B\\ 0 &0 \end{sbmatrix} \right) = W\left(  \begin{sbmatrix} 0&\|B\|\\ 0 &0 \end{sbmatrix} \right)$ is a disc centred at 0 of radius $\|B\|/2$.
Hence $W(T) = \alpha + \frac{\|B\|}2 \ol{\bb D}$, and therefore $w(T) = |\alpha| + \frac12 \|B\|$.
\end{proof}

\begin{thm} \label{T:main}
Let $\cl A$ be a unital operator algebra, and let $\Phi$ be a unital completely bounded homomorphism. Then
\[ \|\Phi\|_{wcb} = \frac12 \big( \|\Phi\|_{cb} + \|\Phi\|_{cb}^{-1} \big) .\]
\end{thm}

\begin{proof}
Let $C = \|\Phi\|_{cb}$. 
By Paulsen's similarity theorem \cite{Paulsen84}, there is an invertible operator $S$ so that
$\Ad S \circ \Phi$ is completely contractive and $\|S\|\,\|S^{-1}\| = C$. 
(Here $\Ad S(T) = STS^{-1}$.)
Let $A \in \cl M_n(\cl A)$ with $\|A\|=1$. 
Then $T := (\Ad S\circ\Phi)^{(n)}(A)$ satisfies $\|T\|\le1$ and $\Phi(A) = \Ad S^{-1(n)}(T)$. 
Hence by Lemma~\ref{L:num_range}, $w(\Phi^{(n)}(A)) \le \frac12(C+C^{-1})$.
Thus  
\[ \|\Phi\|_{wcb} \le \frac12(C+\frac1C) .\]

Conversely, suppose that $A \in \cl M_n(\cl A)$ with $\|A\|=1$ such that $\|\Phi^{(n)}(A)\| > C-\ep$ for some $\ep>0$.
Define $B \in \cl M_{2n}(\cl A)$ by
\[ B = \begin{bmatrix}C^{-1} I_n&(1-C^{-2})A\\0&C^{-1} I_n \end{bmatrix} .\]
Then by Lemma~\ref{L:2x2}, $\|B\|=1$.  Moreover by the second part of that lemma,
\begin{align*}
 \|\Phi\|_{wcb} &\ge w(\Phi^{(2n)}(B)) \\
 &= w\Big(  \begin{bmatrix}C^{-1} I_n&(1-C^{-2})\Phi^{(n)}(A)\\0&C^{-1} I_n \end{bmatrix} \Big) \\
 &> C^{-1} + \frac12 (1-C^{-2}) (C-\ep) \\
 &> \frac12(C + \frac1C) - \frac\ep2. 
\end{align*}
As $\ep>0$ was arbitrary, we obtain 
\[ \|\Phi\|_{wcb} = \frac12(C+\frac1C) = \frac12 \big( \|\Phi\|_{cb} + \|\Phi\|_{cb}^{-1} \big) . \qedhere\]
\end{proof}
\begin{remark} Inverting the above function shows that for a unital homomorphism $\Phi$,
\[ \|\Phi\|_{cb} = \|\Phi\|_{wcb} + \sqrt{\|\Phi\|_{wcb}^2 - 1}.\]
\end{remark}

\section{Consequences}

As an immediate application, we obtain the second theorem stated in the introduction.
Note that convexity of $K$ is not required.

\begin{thm} \label{T:comp_spectral}
Let $C\ge1$ and set $C'=\frac12(C+C^{-1})$.
A compact subset $K\subset \bb C$ is a complete $C$-spectral set for $T\in \cl B(\cl H)$ 
if and only if it is a complete $C'$-numerical radius set for $T$.
\end{thm}

\begin{proof}
If $K$ is a complete $C$-spectral set for $T$, then the map $\Phi_T(f) = f(T)$ for $f \in R(K)$ has $\|\Phi_T\|_{cb} \le C$.
Hence by Theorem~\ref{T:main},
\[ \|\Phi_T\|_{wcb} = \frac12 \big( \|\Phi_T\|_{cb} + \|\Phi_T\|_{cb}^{-1} \big) \le \frac12(C+C^{-1}) = C' . \]
Thus $K$ is a complete $C'$-numerical radius set for $T$.

Conversely, since $\|A\| \le 2 w(A)$, if $K$ is a complete $C'$-spectral set for $T$, it follows that $\Phi$ is completely bounded.
Then 
\[ \frac12(C+C^{-1}) = C' \ge \|\Phi\|_{wcb} =  \frac12 \big( \|\Phi\|_{cb} + \|\Phi\|_{cb}^{-1} \big) \]
implies that $\|\Phi_T\|_{cb} \le C$. So $K$ is a complete $C$-spectral set for $T$.
\end{proof}

We apply this to the family of $C_\rho$-contractions. For these operators, the set $K$ is the unit disc.

\begin{cor} \label{C:rho}
Suppose that $T$ is a $C_\rho$-contraction for $\rho \ge1$. 
If $F:\bb D \to M_n$ is a matrix polynomial $($or has coefficients in $A(\bb D))$, then 
\[ w(F(T)) \le \frac12(\rho+\rho^{-1}) \|F\|_\infty .\]
\end{cor}

\begin{proof}
By \cite[Theorem 2]{OA}, there is an invertible operator $S$ such that $\|S^{-1}TS\| \le 1$ and $\|S\|\,\|S^{-1}\| \le \rho$.
After scaling, we may suppose that $\|F\|_\infty = 1$.
Then by the generalized von Neumann inequality, we have 
\[1 \ge \| F(S^{-1}TS) \| = \|(S^{-1}\otimes I_n) F(T) (S \otimes I_n) \| .\]
Now an application of Lemma~\ref{L:num_range} yields the conclusion.
\end{proof}

The case $\rho=2$ includes all operators $T$ with $w(T) \le1$. 
This provides a matrix polynomial version of Drury's scalar inequality \cite{Drury}. 

\begin{cor}
Suppose that $T$ has $w(T) \le 1$. 
If $F:\bb D \to M_n$ is a matrix polynomial $($or has coefficients in $A(\bb D))$, then 
\[ w(F(C)) \le \frac54 \|F\|_\infty .\]
\end{cor}

\begin{remark}
Note that the class of operators which have the disc an a complete $2$-spectral set contains many operators which do not have numerical radius 1.
For example, let 
\[ T =  \begin{bmatrix}1/2 &3/2\\0&1/2 \end{bmatrix} \qand  S =  \begin{bmatrix}2 &0\\0& 1 \end{bmatrix} .\]
Then $\|S^{-1}TS\| = 1$ and $\|S\|\,\|S^{-1}\| = 2$ but $w(T) = 5/4$.
\end{remark}

As we mentioned in the introduction, Crouzeix showed \cite{Crouzeix} that the numerical range $W(T)$ 
is a complete $C$-spectral set for $T$ for a universal constant $C < 11.08$.
Crouzeix conjectures \cite{Crouzeix2} that the optimal constant is $2$, which is the case for a disc by \cite{OA}. 
The following are immediate from Theorem~\ref{T:comp_spectral}.

\begin{cor}
Let $T$ be a bounded operator on $\cl H$. Suppose that $W(T)$ has a complete Crouzeix constant of $C$, and let $C'= \frac12(C+C^{-1})$. 
If $F:W(T) \to M_n$ is a matrix polynomial $($or has coefficients in $A(W(T)))$, then 
\[ w(F(T)) \le C' \|F\|_{W(T)} .\]
\end{cor}

In particular, the constant $C'=5.6$ is valid. 

\begin{cor} Let $T$ be a bounded operator on $\cl H$.  Then $W(T)$ is a complete $2$-spectral set for $T$ if and only if 
\[w(F(T)) \le \frac54 \|F\|_{W(T)}\]
for every matrix polynomial $F$.
\end{cor}

Thus Crouzeix's conjecture is true for the norm case if and only if the above $5/4$'s inequality holds for every operator $T$. Also, we know that $2$ and $\frac54$ are the best possible constants in each case.

\bigskip

{\em Acknowledgment.} The research was conducted while the third author was
visiting the Institute for Quantum Computing at the University of Waterloo.
He gratefully acknowledges the hospitality of many at the University of Waterloo, including
his gracious host Vern Paulsen.

\end{document}